\documentclass[a4paper, 12pt]{amsart}
\usepackage{amsmath,mathrsfs,amsthm}
\usepackage{txfonts,rotating,tikz}
\usetikzlibrary{intersections,shapes.arrows,calc}

\usepackage[all,cmtip]{xy}
\usepackage{amssymb}
\usepackage{amsfonts}
\usepackage{amscd}
\usepackage{color} 
\usepackage[utf8]{inputenc}
\usepackage{hyperref}
\usepackage{fullpage}
\usepackage{epic}
\usepackage{mathtools}
\usepackage{tikz}
\usepackage{tikz-cd}
\usetikzlibrary{chains}
\usepackage[thinlines]{easytable}
\usepackage{array}
\usepackage{ctable}
\usepackage{pbox}
\usepackage[usestackEOL]{stackengine}
\usepackage[sort, numbers]{natbib}
\tikzcdset{scale cd/.style={every label/.append style={scale=#1},
 cells={nodes={scale=#1}}}}

\newcolumntype{?}{!{\vrule width 1pt}}

\newtheorem*{theorem*}{Theorem}

\newtheorem{theorem}{Theorem}[section]

\newtheorem{lemma}[theorem]{Lemma}
\newtheorem{proposition}[theorem]{Proposition}
\newtheorem{corollary}[theorem]{Corollary}

\newtheorem{remark}[theorem]{Remark}

\theoremstyle{definition}
\newtheorem{definition}[theorem]{Definition}

\newtheorem{example}[theorem]{Example}

\newcommand{\be}{\begin{equation}}
	\newcommand{\ee}{\end{equation}}
\newcommand{\bt}{\begin{theorem}}
	\newcommand{\et}{\end{theorem}}
\newcommand{\bd}{\begin{definition}}
	\newcommand{\ed}{\end{definition}}
\newcommand{\bp}{\begin{proposition}}
	\newcommand{\ep}{\end{proposition}}
\newcommand{\bl}{\begin{lemma}}
	\newcommand{\el}{\end{lemma}}
\newcommand{\bco}{\begin{corollary}}
	\newcommand{\eco}{\end{corollary}}
\newcommand{\br}{\begin{remark}}
	\newcommand{\er}{\end{remark}}
\newcommand{\bex}{\begin{example}}
	\newcommand{\eex}{\end{example}}
\newcommand{\ben}{\begin{enumerate}}
	\newcommand{\een}{\end{enumerate}}
\newcommand{\bc}{\begin{cases}}
	\newcommand{\ec}{\end{cases}}

\newcommand{\bpf}{\begin{proof}}
	\newcommand{\epf}{\end{proof}}

\newcommand{\bma}{\begin{bmatrix}}
	\newcommand{\ema}{\end{bmatrix}}

\DeclareMathOperator{\Spec}{\mathrm{Spec}}

\newcommand{\ms}{\mathscr}
\newcommand{\mb}{\mathbb}

\newcommand{\mf}{\mathfrak}

\DeclareMathOperator{\Gr}{\mathrm{Gr}}

\newcommand{\Lie}{\mathrm{Lie}}
  
\newcommand{\FM}{\mathrm{Gr}^{\rm FM} }

\newcommand{\la}{\langle}  
\newcommand{\ra}{\rangle}

\newcommand{\beqn}{\begin{equation}}
	\newcommand{\eeqn}{\end{equation}}

\newcommand{\arxiv}[1]{\href{https://arxiv.org/abs/#1}{\texttt{arXiv:#1}}}

\makeatletter
\@namedef{subjclassname@2020}{%
  \textup{2020} Mathematics Subject Classification}
\makeatother

\begin{document}
	\title{Tangent spaces of spherical Schubert varieties and counterexamples to the reducedness conjecture}
	\author{Marc Besson}
	\address[M.Besson]{Yau Mathematical Sciences Center, Tsinghua University, Haidian District, Beijing, 100084, China}
\email{bessonm@tsinghua.edu.cn}
	\author{Jiuzu Hong}
	\address[J.Hong]{Department of Mathematics, University of North Carolina at Chapel Hill, Chapel Hill, NC 27599-3250, U.S.A.}
	\email{jiuzu@email.unc.edu}
	\author{Huanhuan Yu}
	\address[H.Yu]{Department of Mathematics, Jiangxi University of Finance and Economics, Nanchang, 330032, China}
	\email{yuhuanhuan@jxufe.edu.cn}
	
	\begin{abstract}
	
		Given a simply-connected simple algebraic group $G$, we determine the tangent space of any Finkelberg-Mirkovi\'c Schubert scheme at the base point of the affine Grassmannian of $G$. As a consequence, we exhibit non-reduced  Finkelberg-Mirkovi\'c Schubert schemes when $G$ is of type $E_6,E_7$ and $E_8$.		
	\end{abstract}
	\maketitle
	\section{Introduction}
	Let $G$ be a simply-connected simple algebraic group over $\mathbb{C}$ and let $\Gr_G$ denote the affine Grassmannian of $G$. For any dominant coweight $\lambda$,  let $\Gr_{\leq \lambda}$ denote the associated spherical Schubert variety, which is the reduced closure of the $L^+G$-orbit associated to $\lambda$.    On the other hand, Finkelberg and Mirkovi\'c  proposed in \cite{FM} a modular description of $\Gr_{\leq \lambda}$ . It is easy to verify that this moduli space is set-theoretically supported on the Schubert variety, but it is not at all clear that this moduli space is a reduced scheme. We refer to this moduli space as a Finkelberg-Mirkovi\'c Schubert scheme (abbreviated as FM Schubert scheme), denoted by $\Gr^{\rm FM}_{\leq \lambda}$.  Since then, it has been a long-standing open question whether FM Schubert schemes are reduced, that is  whether they agree with the usual spherical Schubert varieties.  This was explicitly formulated as a conjecture in \cite{KMW}; see also \cite[Conjecture 2.14, conjecture 2.10]{KWWY}, and the discussions in \cite[Remark 2.1.7]{Zh}. Subsequently, Kamnitzer-Muthiah-Weekes-Yacobi proved that the conjecture holds for type $A$  in \cite{KMWY}.  As explained in \cite{KWWY}, the validity of this conjecture also has significant applications to the quantization of affine Grassmannian slices and Yangians. 

The singularities of spherical Schubert varieties have been studied extensively in the literature. Notably, Evens-Mirkovi\'c proved in \cite{EM} that the smooth locus of $\Gr_{\leq \lambda}$ is its big Schubert cell; see a different and elegant proof by Malkin-Ostrik-Vybornov \cite{MOV}.  Recently, this result has been extended to twisted affine Grassmannians in \cite{BH} by the first two authors, which confirms a conjecture of Haines-Richarz \cite{HR}. Later, Pappas-Zhou \cite{PZ} gave a different proof by establishing a lower bound of the dimension of the tangent space of $\Gr_{\leq \lambda}$.  Subsequently, Kisin-Pappas-Zhou \cite{KPZ} gave an upper bound of the dimension of the tangent space of $\Gr^{\rm FM}_{\leq \lambda}$. Thus, this provides an upper bound for the tangent space of $\Gr_{\leq \lambda}$. For classical groups, in many cases they show that the two dimensions coincide. 

 Inspired by the work \cite{KPZ},  we explicitly compute the dimensions of the tangent spaces of the spherical Schubert variety $\Gr_{\leq \theta^\vee}$ and the FM Schubert scheme $\Gr^{\rm FM}_{\leq \theta^\vee}$ at the base point $e$, where $\theta^\vee$ is the quasi-minusucle coweight, equivalently, the
  coroot of the highest root $\theta$ of $G$.  We show that in Corollary \ref{tan_FM_theta},  except in type $E_8$, these two tangent spaces agree and have the same dimension as $G$; in the case of $E_8$, however  we prove that 
 $$\dim T_e \Gr_{\leq \theta^\vee}=\dim G, \quad \dim T_e \Gr^{\rm FM}_{\leq \theta^\vee}=2\dim G.$$  
 In fact, more generally in Theorem \ref{thm_tangent_base} we determine the tangent space of any FM Schubert scheme at the base point. 
 As a consequence, we conclude that the FM Schubert scheme $\Gr^{\rm FM}_{\leq \theta^\vee}$ is not reduced when $G$ is of type $E_8$.  After discovering this striking example,  we also found further examples using the same methods when $G$ is of type $E_6$ and $E_7$;  see Section \ref{Section 4.3}.  
 
 After communicating our counterexamples to Alex Weekes, he informed us that, using our tangent space calculations,  these examples also lead to a failure of the reducedness conjecture for the truncated shifted Yangians.

 \vspace{0.2cm}

{\bf Acknowledgement}:  We would like to thank R.\,Travkin for helpful discussions, especially on the proof of Lemma \ref{lem_4.4}. We also thank S.\,Kumar for valuable discussions and references.  In addition, we thank A.\,Weekes and T.\,Haines for careful reading of a preliminary version and helpful comments. 

%
	\section{Notation and preliminaries}
	\subsection{}
	Let $G$ be a simple algebraic group over $\mb{C}$. We fix a maximal torus $T$ and a Borel subgroup $B$ containing $T$ in $G$. Associated to $T$ we have weight and coweight lattices $X^*(T)$ and $X_*(T)$ respectively. We write in gothic letters the respective Lie algebras, so we have inclusions of Lie algebras $\mathfrak{t} \rightarrow \mathfrak{b} \rightarrow \mathfrak{g}$. 
	
	Let $\Phi$ be the set of roots with respect to the maximal torus $T$.  Let $W$ denote the Weyl group $N_G(T)/T$.  Via the choice of Borel subgroup $B$, we have the set $\Phi^+$ of positive roots, and the set $X^*(T)^+$ of dominant weights; dually,  we also have the set $\Phi^\vee$ of coroots  and the set $X_*(T)^+$ of  dominant coweights. For any $\alpha\in \Phi$, we denote $\mf{g}_\alpha$ the corresponding root subspace in $\mf{g}$, and we choose a root vector $x_\alpha\in \mf{g}_\alpha$. For each $
	\alpha\in \Phi$, we denote by $\alpha^\vee\in \Phi^\vee$ the associated coroot. We denote $
	\theta$ the highest root in $\Phi$, and $\theta^\vee$ the coroot of $\theta$.  
	
	 For any $\eta \in X^*(T)^+$ we write $V_{\eta}$ for the associated irreducible highest-weight representation of $G$ of highest weight $\eta$. We denote the dimension of $ V_\eta$ by $d_\eta$. 
	
	\subsection{}
	For any simple algebraic group $G$ over $\mathbb{C}$, let $LG$ denote the loop group functor and let $L^+G$ denote the jet group functor. These are defined as follows:  for any $\mb{C}$-algebra $R$, 
\[LG( R) =G(R((t))), \quad L^+G(R)=G(R[[t]]) .\]
Let $\Gr_G$ be the affine Grassmannian of $G$, which is defined as the quotient  $\Gr_G= LG/L^+G$ in the \'etale topology.  The affine Grassmannian $\Gr_G$ is an ind-scheme of finite type over $\mb{C}$.
	 For each $\lambda\in X_*(T)$, we can attach a point $t^\lambda\in LG$, and also an associated point $e_\lambda \in \Gr_G$.  Let $\Gr_\lambda$ denote the $L^+G$-orbit of $e_\lambda$ in $\Gr_G$.  Then, we have the Cartan decomposition 
	\[\mathrm{Gr}_G= \bigsqcup_{\lambda \in X_*(T)^+} \Gr_\lambda.\]
We denote by $\Gr_{\leq \lambda}$ the reduced closure of $\Gr_\lambda$ in $\Gr_G$. The variety $\Gr_{\leq \lambda}$ is called a spherical Schubert variety.  We have $\Gr_\mu\subset \Gr_{\leq \lambda}$ if and only if $\mu \preceq \lambda$,  where $\preceq$ is the standard partial order on the dominant coweights. 
	
Let $\mb{C}[\epsilon]$ denote the ring of dual numbers, i.e. $\epsilon$ is subject to the relation $\epsilon^2=0$. 	For any group functor $H$ over $\mb{C}$, the Lie algebra ${\rm Lie} (H)$ of $H$ is defined to be the kernel of the following evaluation map:  
	\[    H(\mb{C}[\epsilon] )\to H(\mb{C})  .  \]
Then, we have natural isomorphisms:  
\[  \Lie (LG)\simeq \mf{g}((t)), \quad \Lie(L^+G)\simeq \mf{g}[[t]].   \]	

	Let $L^-G$ be the group functor given by  
	$$L^-G( R)= G(R[t^{-1}]).$$
	In fact, $L^-G$ is represented by an ind-affine group scheme. Let $L^{--}G$ be the kernel of the evaluation map $L^-G\to G$. We have, 
	\[  \Lie( L^-G)\simeq \mf{g}[t^{-1}], \quad  \Lie( L^{--}G) \simeq t^{-1}\mf{g}[t^{-1}] .   \]
	By \cite[Corollary 3]{Fa}, there is an open immersion, by taking the orbit through $e \in \mathrm{Gr}_G$:
	$$L^{--}G \to \Gr_G.$$
	This induces an isomorphism of $G\times \mb{G}_m$-representations: 
	\begin{equation}
	\label{eq_tan_Gr}
	t^{-1}\mathfrak{g}[t^{-1}]\simeq T_e\mathrm{Gr}_G,   \end{equation}
	where the action of $\mb{G}_m$ is induced from the rotation torus acting on $LG$ and $\Gr_G$. 
	For later use, we decompose it explicitly as:
	\begin{equation}\label{eq_tan_Gr}
	T_e\mathrm{Gr}_G\simeq \bigoplus_{k\geq 1} \mf{g}t^{-k}\simeq \Big(\bigoplus_{\alpha\in \Phi, k\geq 1} \mf{g}_{\alpha}t^{-k}\Big)\oplus    \Big(\bigoplus_{k\geq 1} \mf{t} \,t^{-k}\Big).\end{equation} 

	\section{Tangent spaces of affine Schubert Varieties: Polo's criterion}
From now on, we assume that $G$ is a simply-connected simple algebraic group over $\mb{C}$.
\subsection{}
	We make passing use of the affine Kac-Moody algebra $\hat{\mathfrak{g}}$ associated to $\mathfrak{g}$; this is a central extension of the Lie algebra $ \mathfrak{g}((t))  \oplus \mathbb{C}d $ by the center $\mathbb{C}K$, where $d$ is the scaling element,  cf.\,\cite[\S 7]{Kac}.  The Lie subalgebras $\mf{g}[[t]]$ and $\mf{g}[t^{-1}]$ in $\mf{g}((t))$ have natural splittings in $\hat{\mf{g}}$. Thus,  they are all naturally Lie subalgebras in $\hat{ \mf{g} }$. Note that the scaling element $d$ corresponds to the rotation action of $\mb{G}_m$ on the loop group $LG$ and the loop algebra $\mf{g}((t))$. 
	
	 Set $\hat{\mf{t}}:= \mf{t}\oplus \mb{C}K \oplus \mb{C}d$.  The fundamental weight associated to the affine node in the affine Dynkin diagram of $\mf{g}$ is an element $\Lambda_0\in (\hat{\mf{t}})^*$ defined by 
	\[  \Lambda_0|_{\mf{t}\oplus \mb{C}d }=0, \quad \Lambda_0(K)=1 .  \]

\subsection{}

	In the setting of Schubert varieties of finite type,  Polo's criterion describes the Zariski tangent space of Schubert variety in terms of Demazure modules, see \cite{Po}.  This criterion has been extended to the setting of Kac-Moody Schubert varieties by Kumar, cf.\,\cite[Ex.12.1.E (9)]{Ku}.

	 Let $\mathscr{L}$ denote the level one line bundle on $\Gr_G$.  Denote by $\mathscr{H}_{\Lambda_0}$ the integrable highest weight representation of $\hat{\mathfrak{g}}$ with highest weight $\Lambda_0$.  By the affine analogue of the Borel-Weil theorem (cf.\,\cite[Theorem 8.3.11]{Ku}), there is an isomorphism of $\hat{\mf{g}}$-representations: 
	\begin{equation}
	\Gamma(\Gr_G, \mathscr{L})^*
 \simeq \mathscr{H}_{\Lambda_0}.   \end{equation}
For any $\lambda\in X_*(T)^+$, let $v_{t^\lambda(\Lambda_0)}$ be an extremal weight vector in $\mathscr{H}_{\Lambda_0}$, with the weight $t^\lambda(\Lambda_0)$ given by
\begin{equation}
\label{eq_extrem_wt}
 t^\lambda(\Lambda_0)= \Lambda_0- \iota(\lambda)- \frac{1}{2}(\lambda, \lambda) \delta  , \end{equation}
where $(\cdot,\cdot)$ is the normalized Killing form and $\iota: X_*(T)\to \mf{t}^*$ is the map induced from $(\cdot,\cdot)$.

We define the level one affine Demazure module $D(1,\lambda)$ as follows,
\begin{equation} D(1,\lambda):=  U(\mf{g}[[t]] )\cdot v_{t^\lambda(\Lambda_0)} \subset \ms{H}_{\Lambda_0}.  \end{equation}
Then, by \cite[Theorem 8.2.2]{Ku}, we have an isomorphism of $\mf{g}[[t]]$-modules:
\begin{equation}\Gamma(\Gr_{\leq \lambda}, \mathscr{L})^*\simeq  D(1,\lambda) . \end{equation}

	We state a special case of Polo's criterion for spherical Schubert varieties,  cf.\,\cite[Corollary 4.3]{HLR}:
	\begin{theorem}\label{KMPolo}
	For $\lambda\in X_*(T)^+$,  we have an isomorphism
	\[T_{e} \mathrm{Gr}_{\leq \lambda}\simeq \big\{ X\in t^{-1}\mf{g}[t^{-1}]  \, |\,  X.v_{ \Lambda_0} \in D(1,\lambda) \big\}\]
	as representations of $G\times \mb{G}_m$.
	\end{theorem}
	
	The following proposition is known to experts, cf.\,\cite[Proposition 6.1]{HLR}. For the convenience of readers, we provide two self-contained proofs with different perspectives. 
	
	\begin{proposition}
	\label{prop_tan_Sch}
	Let $G$ be a simple linear algebraic group over $\mathbb{C}$. Then \[T_e\mathrm{Gr}_{\leq \theta^{\vee}} \simeq  \mathfrak{g}t ^{-1} \subset t^{-1} \mathfrak{g}[t^{-1}]\]
	as representations of $G$, 
	where $\theta$ is the highest root of $G$ and $\theta^\vee$ is the coroot of $\theta$.
	\end{proposition}
	
	\begin{proof}
	The affine Demazure module $D(1,\theta^\vee)$ is the $U(\mf{g}[t])$-submodule of $\mathscr{H}_{\Lambda_0}$ generated by the extremal weight vector $v_{t^{\theta^{\vee}}\Lambda_0}$. By (\ref{eq_extrem_wt}),  the $\mathfrak{t}$-weight of $v_{t^{\theta^{\vee}}\Lambda_0}$ is $-\theta$, and the rotation degree is $-\frac{1}{2}(\theta^\vee, \theta^\vee)=-1$.  In fact, one may see easily that, up to a nonzero scalar,
	$$v_{t^{\theta^{\vee}}\Lambda_0}= (x_{-\theta}t^{-1})\cdot  v_{\Lambda_0},$$
	which is a lowest weight vector with respect to the action of $\mf{g}$. 
	By the Serre relations for integrable highest weight representation of $\hat{\mf{g}}$, there is an isomorphism:
	\begin{equation}
	\label{eq_serre}
	\mathscr{H}_{\Lambda_0}\simeq \frac{\mathrm{U}(t^{-1}\mathfrak{g}[t^{-1}])}{\mathrm{U}(t^{-1}\mathfrak{g}[t^{-1}])(x_{\theta}t^{-1})^2}\simeq \mathbb{C}v_{\Lambda_0}\oplus \mathfrak{g} t^{-1}.v_{\Lambda_0}\oplus\cdots\end{equation}
respecting the actions of $\mf{g}$ and the scaling element $d$, 
	where $\mathbb{C}v_{\Lambda_0}$ has $d$-weight $0$, $ \mathfrak{g}t^{-1}$ has $d$-weight $-1$ and all other terms have lower $d$-weight.  Thus, the $U(\mf{g}[t])$-action on $v_{t^{\theta^{\vee}} \Lambda_0}$ generates a copy of the adjoint representation $\mathfrak{g}$ with $d$-weight $-1$ as well as the highest weight line $\mb{C} v_{\Lambda_0}$:
	\[D(1,\theta^\vee)=\mathbb{C} v_{\Lambda_0} \oplus \mathfrak{g}t^{-1}.\]
Then, for any $X\in \mf{g}t^{-1}$, $X\cdot v_{\Lambda_0}\in D(1, \theta^\vee)$. Moreover, by applying PBW theorem for $U(t^{-1}\mf{g}[t^{-1}])$ in (\ref{eq_serre}), for any nonzero vector $X\in \mf{g}t^{-k}$ with $k\geq 1$,  $X\cdot v_{\Lambda_0}\not=0$. Then, by Theorem \ref{KMPolo}, we can conclude that   
	\[T_e \mathrm{Gr}_{\leq \theta^{\vee}} \simeq   \mathfrak{g}t^{-1}.\]
		\end{proof}

In the following, using the work \cite{MOV} we provide a different proof of Proposition \ref{prop_tan_Sch}. 	
	\begin{proof}
	Let $L^{--}G\cdot e$ be the orbit of $L^{--}G$ at $e$. Then, $(L^{--}G\cdot e)\cap \Gr_{\leq \theta^\vee}$ is an open neighborhood of $e$ in $\Gr_{\leq \theta^\vee}$.  Let $\overline{\mb{O}}_{\rm min}$ be the minimal nilpotent variety in $\mf{g}$, which is the closure of $G$-orbit at $x_\theta$.  By \cite[Lemma 2.10]{MOV},  the map 
	$$\overline{\mb{O}}_{\rm min}\to(L^{--}G\cdot e)\cap  \Gr_{\leq \theta^\vee}  $$
given by $x\mapsto {\rm exp}(xt^{-1})\cdot e$, is a $G$-equivariant isomorphism of varieties.  Thus, we have 
\begin{equation}
\label{eq_tan_min}
  T_0 ( \overline{\mb{O}}_{\rm min} )\simeq  T_e  \Gr_{\leq \theta^\vee}.   \end{equation}
By the argument in \cite[\S 2.9]{MOV}, $T_0 ( \overline{\mb{O}}_{\rm min} )\simeq \mf{g}$.  Then, it is easy to see the image of the map (\ref{eq_tan_min}) lands in $\mf{g}t^{-1}$ via the identification (\ref{eq_tan_Gr}).
	\end{proof}
	
	\section{Tangent spaces of FM Schubert schemes}
	\subsection{ }
 For any $\mb{C}$-algebra $R$, set
	\[ D_R:= \mathrm{Spec} R[[t]], \quad   D_R^*:= \mathrm{Spec} R((t)).   \]
	We write $\mathscr{E}$ for a $G$-torsor and $\mathring{\mathscr{E}}$ for the trivial $G$-torsor. The affine Grassmannian $\Gr_G$  can be described as the following moduli space: 
	\[\Gr_G(R)=\{(\mathscr{E}, \phi)\,|\, \mathscr{E} \text{ is a $G$-torsor over } D_R, \phi: \mathscr{E}|_{D^*_R} \xrightarrow{\simeq} \mathring{\mathscr{E}}|_{D^*_R}\}  .\]
Since $G$ is simple,  $\Gr_G$ is reduced.  	
	In light of the moduli description of $\Gr_G$, it is natural to also seek moduli interpretations of $\Gr_{\leq \lambda}$, and the first such proposal was given by Finkelberg-Mirkovi\'c \cite{FM}.
		
	For any $G$-torsor $\ms{E}$ over $D_R$,  
	for each dominant weight $\eta  \in X^*(T)^+$,  we write $\mathscr{E}(\eta)$ for the associated vector bundle $\mathscr{E} \times^{G} V_{\eta}$ over $D_R$.  We naturally have isomorphisms of $R[[t]]$-modules:
	$$\Gamma(D_R, \mathring{\mathscr{E}}(\eta)  )\simeq V_\eta\otimes R[[t]], \quad \Gamma(D^*_R, \mathring{\mathscr{E}}(\eta)|_{D^*_R})\simeq V_\eta\otimes R((t)).$$
For any point $(\ms{E}, \phi)\in \Gr_G(R)$,  let $L_{\eta,R}$ denote the lattice  $\phi ( \Gamma(D_R, \ms{E}(\eta))  )$ in $V_\eta\otimes R((t))$.

 By Beauville–Laszlo theorem, the definition of FM Schubert Scheme in \cite[\S 2.3]{KMW} is equivalent to the following one; see also \cite[Remark 2.1.7]{Zh}.
	\begin{definition}
	\label{def_FM_sch}
	Let $\Gr_{\leq \lambda}^{\rm FM}$ be the representable subfunctor of $\Gr_G$ whose $R$-points are $(\mathscr{E}, \phi) \in \Gr_G(R)$ such that for any $\eta\in X^*(T)^+$
	\begin{equation}
	\label{eq_def_FM}
	   V_\eta\otimes R[[t]]   \subseteq  t^{-\langle \lambda, \eta \rangle}  L_{\eta, R}.\end{equation}
	\begin{remark}
	By taking the dual, the condition (\ref{eq_def_FM}) applies to $-w_0(\eta)$ yields to the condition: $t^{-\la \lambda, w_0(\eta) \ra} L_{\eta,R}\subset V_\eta\otimes R[[t]]$, where $w_0$ is the longest element of the Weyl group $W$.  Our condition (\ref{eq_def_FM}) differs by a twist of $-w_0$ from that in other references. This condition is exactly compatible with the loop group action on the lattices in (\ref{eq_Plu_act} ).
		\end{remark}
	
	\end{definition}
It is known that $\Gr_{\leq \lambda}^{\rm FM}$ is a closed subscheme of finite type in $\Gr_G$, and the scheme $\Gr_{\leq \lambda}^{\rm FM}$ is related to the spherical Schubert variety $\mathrm{Gr}_{\leq \lambda}$ as follows: there is a closed immersion $\mathrm{Gr}_{\leq \lambda} \rightarrow \mathrm{Gr}_{\leq \lambda}^{\rm FM}$
	such that 
	\[(\mathrm{Gr}_{\leq \lambda}^{\rm FM})_{\rm red} \simeq \mathrm{Gr}_{\leq \lambda},\]
cf.\,\cite[Proposition 2.1.4]{Zh}.	It is a long-standing question if $\mathrm{Gr}_{\leq \lambda}^{\rm FM} \simeq \mathrm{Gr}_{\leq \lambda}$.
	
Following \cite{Hai,HJ}, the affine Grassmannian $\Gr_G$ admits a Pl\"ucker description:  the set $\Gr_G(R)$ can be identified with the set of collections 
\[    \{   L_{\eta,R}  \}_{\eta\in X^*(T)^+ }  \]
consisting of $R[[t]]$-lattices $L_{\eta, R}\subset V_\eta \otimes R((t))$, satisfying the natural Pl\"ucker relations as in \cite[Proposition 6.1]{Hai}.  From this perspective, we can describe the action of $G(R((t)) )$ on $\Gr_G(R)$ as follows, for any $g\in G(R((t)) )$,  
\begin{equation}
\label{eq_Plu_act}
 g\cdot  \{   L_{\eta,R}  \}_{\eta\in X^*(T)^+ } = \{  g\cdot L_{\eta,R}   \}_{\eta\in X^*(T)^+ } , \end{equation}
	where $g\cdot L_{\eta,R} $ is another $R[[t]]$-lattice in $V_\eta \otimes R((t))$ translated by $g$, with the action induced from the representation of $G$ on $V_\eta$. 
	
For convenience, set 
	$$L^0_{\eta, R}:= V_\eta\otimes R[[t]] .$$
	  Then, $\{ L^0_{\eta, R} \}_{\eta\in X^*(T)^+}$ corresponds to the base point $e$ in $\Gr_G(R)$.  

Recall that $T_e\Gr_G$ is the fiber of the evaluation map
\[ {\rm ev}: \Gr_G(\mb{C}[\epsilon] )\to  \Gr_G(\mb{C}) \]
over the base point $e\in \Gr_G(\mb{C})$.
In terms of the Pl\"ucker moduli interpretation,  $T_e\Gr_G$ can be identified with the set of those collections $\{ L_{\eta, \mb{C}[\epsilon]}  \}_{\eta\in X^*(T)^+}$ such that for each $\eta\in X^*(T)^+$, 
$$L_{\eta, \mb{C}[\epsilon]}/\epsilon L_{\eta, \mb{C}[\epsilon]}\simeq V_\eta\otimes \mb{C}[[t]] ,$$
 and these isomorphisms are compatible with the Pl\"ucker relations for $\{ L_{\eta,\mb{C}[\epsilon]} \}$ and $\{V_\eta\otimes \mb{C}[[t]] \}$. 

Since the orbit map $L^{--}G\to \Gr_G$ through the base point $e$ is an open immersion and $\Lie (L^{--}G)\simeq t^{-1}\mf{g}[t^{-1}] $, the moduli interpretation of the isomorphism (\ref{eq_tan_Gr}) is the following:
\begin{equation}
\label{eq_tan_sp}
  T_e\Gr_G = \Big\{  \big\{(1+X \epsilon )\cdot L^0_{\eta, \mb{C}[\epsilon]}    \big\}_{\eta\in X^*(T)^+}  \,\big |\,    X\in t^{-1}\mf{g}[t^{-1}]      \Big\}     .\end{equation}
  
We have the following lemma which helps to compute the tangent space $T_e\FM_{\leq \lambda}$ of $\FM_{\leq \lambda}$ at $e$. \begin{lemma}
 \label{lem_crit_FM}
 The tangent space $T_e\FM_{\leq \lambda}$ is isomorphic to the subspace of $t^{-1}\mf{g}[t^{-1}]$ consisting of those $X\in t^{-1}\mf{g}[t^{-1}]$ such that for any $\eta\in X^*(T)^+$, 
 \[  L^0_{\eta, \mb{C}[\epsilon]} \subseteq t^{-\la \lambda, \eta \ra} (1+ X\epsilon)\cdot L^0_{\eta, \mb{C}[\epsilon]} ;\]
 equivalently,
 \[ (1-X\epsilon)\cdot   L^0_{\eta, \mb{C}[\epsilon]} \subseteq t^{-\la \lambda, \eta \ra} L^0_{\eta, \mb{C}[\epsilon]}  .\] 
 Furthermore, this is an isomorphism of $G\times \mb{G}_m$-representations.
 \end{lemma} 
  \begin{proof}
  This lemma immediately follows from the description (\ref{eq_tan_sp}) and the definition of  $\FM_{\leq \lambda}$.
  \end{proof}
  
 In \cite[Definition 4.2.6]{KPZ}, the authors gave a definition of FM Schubert scheme with one more condition than Definition \ref{def_FM_sch}, for which we denote by $\Gr^{\rm FM '}_{\leq \lambda}$. In the setting of Definition \ref{def_FM_sch}, the additional condition in \cite{KPZ} is following: for any $\eta\in X^*(T)^+$, 
\begin{equation}
\label{eq_cond_1}
 \wedge^{d_\eta} L_{\eta, R} = \wedge^{d_\eta} ( t^\lambda\cdot  L^0_{\eta ,   R} )   \end{equation}
regarded as $R[[t]]$-submodules in $\wedge^{d_\eta} ( L^0_{\eta,R}[\frac{1}{t}] )$.

 It was shown in  \cite[Proposition 6.4]{Hai} that $(\Gr^{\rm FM '}_{\leq \lambda})_{\rm red}= \Gr_{\leq \lambda}$. For general reductive group $G$, we don't know if $\FM_{\leq \lambda}=\Gr^{\rm FM '}_{\leq \lambda}$. However, by the following lemma, when $G$ is simply-connected and semisimple,  we do have $\FM_{\leq \lambda}=\Gr^{\rm FM '}_{\leq \lambda}$.  \begin{lemma}
 \label{lem_4.4}
When $G$ is simply-connected and semisimple,  for any $R$-point $(\ms{E}, \phi)\in \Gr_G(R)$, the condition (\ref{eq_cond_1}) holds. 
\end{lemma} 
 \begin{proof}
 
 We first observe that $\wedge^{d_\eta} ( t^\lambda\cdot  L^0_{\eta ,   R} ) = \wedge^{d_\eta} L^0_{\eta,R}$, since for any weight $\delta$ in $V_\eta$, $\sum_{w\in W}w (\delta)=0$.  Thus, we are reduced to show $\wedge^{d_\eta} L_{\eta, R}= \wedge^{d_\eta} L^0_{\eta,R}$.
 
 From the proof of \cite[Theorem 1.1.3]{Zh} and \cite[Proposition 1.2.6]{Zh}, for any $R$-point $\Spec R\to \Gr_G$, there is a finitely generated $\mb{C}$-algebra $R'$ such that this $R$-point factor through a $R'$-point $\Spec R'\to \Gr_G$.  Thus, we can assume $R$ is finitely-generated over $\mb{C}$.  On the other hand,   since $G$ is semi-simple, $\Gr_G$ is ind-reduced, i.e. $\Gr_G$ is a union $\bigcup_{\alpha} X_\alpha$ of reduced closed subschemes $X_\alpha$ of finite type.  Thus, $\Spec R\to \Gr_G$ factors through a closed embedding $X\to \Gr_G$, where $X=X_\alpha$ for some $\alpha$.  We are now reduced to show, 
 \[ \wedge^{d_\eta} L_{\eta, X} = \wedge^{d_\eta} ( t^\lambda\cdot  L^0_{\eta ,   X} )   \]
 as $\ms{O}_X[[t]]$-submodules in $\wedge^{d_\eta} L^0_{\eta, X}[\frac{1}{t}]$. 

 Note that $L_{\eta, X}$ contains a lattice $t^{N}L^0_{\eta, X}$ when $N$ is sufficiently large.  Then,  $\wedge^{d_\eta}L_{\eta, X} $ and $\wedge^{d_\eta} L^0_{\eta,X}$ share the $\ms{O}_X[[t]]$-submodule $\wedge^{d_\eta} t^N L^0_{\eta,X}$. Therefore, to show $\wedge^{d_\eta}L_{\eta, X} =\wedge^{d_\eta} L^0_{\eta, X}$, it suffices to show that 
 \begin{equation}
 \label{eq_iso_vb}
    \wedge^{d_\eta}L_{\eta, X} / \wedge^{d_\eta} t^N L^0_{\eta,X}= \wedge^{d_\eta} L^0_{\eta,X}/\wedge^{d_\eta} t^N L^0_{\eta,X} .\end{equation}
They are both vector bundles over $X$ of finite rank. Over every $\mb{C}$-point $x\in X$, it is clear that 
\[  \wedge^{d_\eta} L_{\eta, x} = \wedge^{d_\eta} L^0_{\eta,x},    \]
  cf.\,\cite[Proposition 6.4]{Hai}. Thus, (\ref{eq_iso_vb}) is an identity over every point $x\in X$. It follows that $(\ref{eq_iso_vb})$ is an identity over $X$. 
 \end{proof}

		\subsection{}
		
		Let $\lambda\in X_*(T)^+$ be any nonzero dominant coweight.  Set 
		\begin{equation}\label{eq_min_lam}
		m_\lambda:= \min\big \{ \la \lambda, \eta \ra \,|\,  \eta\in X^*(T)^+ \backslash \{0\} \big \} .\end{equation}
		Since $G$ is simply-connected, $\lambda$ is a summation of positive coroots. Furthermore, to determine $m_\lambda$ it suffices to compute $\la\lambda, \eta \ra$ for all fundamental weights $\eta$.
	
\begin{theorem}
\label{thm_tangent_base}
There is an isomorphism 
\[ T_e\FM_{\leq \lambda} \simeq \bigoplus_{1\leq k \leq m_\lambda}   \mf{g}t^{-k}  \]
as representations of $G\times \mb{G}_m$. 
\end{theorem}	
	\begin{proof}
	Under the identification (\ref{eq_tan_Gr}),  to check $\mf{g}t^{-k}\subset T_e \FM_{\leq \lambda}$, it suffices to verify that $x_{-\theta}t^{-k}\in T_e \FM_{\leq \lambda}$, where $x_{-\theta}$ is a nonzero root vector in $\mf{g}_{-\theta}$. Because the tangent space $T_e \FM_{\leq \lambda}$ is a representation of $G\times \mb{G}_m$ and the adjoint representation $\mf{g}$ of $G$ is irreducible. By Lemma \ref{lem_crit_FM}, it then suffices to check for any nontrivial $\eta\in X^*(T)^+$, 
 \begin{equation}
 \label{eq_lw_crit}
  (1-x_{-\theta}t^{-k}  \epsilon)\cdot L^0_{\eta, \mb{C}[\epsilon]} \subseteq t^{-\la \lambda, \eta \ra} L^0_{\eta, \mb{C}[\epsilon]} . \end{equation}
  When $1\leq k\leq m_\lambda\leq \la \lambda, \eta \ra$, the condition (\ref{eq_lw_crit}) obviously holds for any $\eta\not=0$.  
  
In view of the $G\times \mb{G}_m$-action on $T_e \FM_{\leq \lambda}$,  if $x_{-\theta}t^{-k}\not \in T_e \FM_{\leq \lambda}$, then any nonzero vector in $\mf{g}t^{-k}$ is not contained in $T_e \FM_{\leq \lambda}$.  Let $v_\eta$ be the highest weight vector in $V_\eta$. Since the coefficient of each simple coroot in $\theta^\vee$ is nonzero (cf.\,\cite[Table 2, p.66]{Hu}), we have 
 $\la \eta, \theta^\vee \ra>0$. Then,  
 \[ (x_{-\theta}t^{-k}  \epsilon)\cdot v_\eta= (x_{-\theta}\cdot v_\eta)t^{-k} \epsilon  \not=0.  \]
 When $k> m_\lambda$, there exists a nonzero dominant weight $\eta_0\in X^*(T)^+$ such that $k> \la \lambda, \eta_0\ra$.  Then, 
 the condition (\ref{eq_lw_crit}) does not hold for $\eta_0$.   By Lemma \ref{lem_crit_FM}, $x_{-\theta}t^{-k} \not \in T_e\FM_{\leq \lambda}$.  As a consequence,  when $k> m_\lambda$, $T_e\FM_{\leq \lambda}$ does not contain any nonzero vectors in $\mf{g}t^{-k}$.  Therefore, we have 
 \[ T_e\FM_{\leq \lambda} \simeq \bigoplus_{1\leq k \leq m_\lambda}   \mf{g}t^{-k}  .\]
	\end{proof}

	We now apply this theorem to the quasi-minuscule coweight $\theta^\vee$ of $G$.
	\begin{corollary}
	\label{tan_FM_theta}
	\begin{enumerate}
	\item If $G$ is not of type $E_8$,  then 
	\[  T_e\FM_{\leq \theta^\vee}\simeq \mf{g}t^{-1} . \]
	\item If $G$ is of type $E_8$, then 
	\[  T_e\FM_{\leq \theta^\vee}\simeq \mf{g}t^{-1}\oplus \mf{g}t^{-2} . \]
	\end{enumerate}
	\end{corollary}	
	\begin{proof}
By Theorem \ref{thm_tangent_base}, it suffices to determine the number $m_{\theta^\vee}$. In the following, we follow the  labelling of Dynkin diagram in \cite{Hu}. 
	
	For Part 1).  Note that $\theta^\vee$ is the highest short coroot of $G$.  Then,  from \cite[Table 2, p.66]{Hu}, we see that for all simple root systems except $E_8$ with a chosen base $\Delta$, $\theta^{\vee}= \sum a_i \alpha_i^{\vee}$ always has some $a_i=1$.  Thus, $m_{\theta^\vee}=1$.


For Part 2).  When $G$ is of type $E_8$, the coroot $\theta^{\vee}$ can be expressed as follows: 
\[\theta^{\vee}=2 \alpha_1^{\vee}+3 \alpha_2^{\vee}+4 \alpha_3^{\vee}+6\alpha_4^{\vee}+5\alpha_5^{\vee}+4\alpha_6^{\vee}+3\alpha_7^{\vee}+2\alpha_8^{\vee}.\]
From this formula, it is clear that  $m_{\theta^\vee}=2$.
\end{proof}

Finally, we can conclude the following theorem, which exhibits a nonreduced FM Schubert scheme.
\begin{theorem}
When $G$ is of type $E_8$,  the FM Schubert Scheme $\FM_{\leq \theta^\vee}$ is not reduced. 
\end{theorem}	
\begin{proof}
It follows from Proposition \ref{prop_tan_Sch} and Part 2) of Corollary \ref{tan_FM_theta}.
\end{proof}

\subsection{Counterexamples in type $E_6$ and $E_7$}
\label{Section 4.3}
When $G$ is of type $E_6$,  consider the coweight 
$$\omega_1^\vee+\omega_6^\vee=2\alpha^\vee_1+ 2\alpha_2^\vee+ 3 \alpha^\vee_3+ 4\alpha^\vee_4+ 3\alpha^\vee_5+2\alpha^\vee_6,$$ 
cf.\,\cite[Table 1, p.69]{Hu}, where $\omega^\vee_i$ denote the $i$-th fundamental coweight of $G$.  Thus, $m_{\omega_1^\vee+\omega_6^\vee}=2$ and by Theorem \ref{thm_tangent_base}, 
$$ T_e \FM_{\preceq \omega_1^\vee+\omega_6^\vee} = \mf{g}t^{-1}\oplus \mf{g}t^{-2}.$$
  Since $\theta^\vee\prec \omega_1^\vee+\omega_6^\vee$,  by Proposition \ref{prop_tan_Sch}, 
we have 
$$\mf{g}t^{-1}\simeq  T_e\Gr_{\preceq \theta^\vee}\subset T_e\Gr_{\preceq \omega_1^\vee+\omega_6^\vee}. $$
By Polo criterion (cf.\,Theorem \ref{KMPolo}),  there is an embedding $T_e\Gr_{\preceq \omega_1^\vee+\omega_6^\vee}\hookrightarrow D(1,\omega_1^\vee+\omega_6^\vee)$  of $G\times \mathbb{G}_{\rm m}$-representations.  
 Observe from the decomposition of $D(1, \omega^\vee_1+\omega^\vee_6)$ in \cite[\S5.1, p.25]{BH}, the adjoint representation $\mf{g}\simeq V_{\omega_2}$ appears in $D(1,\omega_1^\vee+\omega_6^\vee)$ with multiplicity one.  Thus, 
 $$T_e\Gr_{\preceq \omega_1^\vee+\omega_6^\vee}\simeq \mf{g} t^{-1}.$$
 We conclude that $\Gr_{\preceq \omega_1^\vee+\omega_6^\vee}\not= \FM_{\preceq \omega_1^\vee+\omega_6^\vee}$.

When $G$ is of type $E_7$,  consider  
$$ \omega_6^\vee=2\alpha^\vee_1+ 3\alpha_2^\vee+ 4 \alpha^\vee_3+ 6\alpha^\vee_4+ 5\alpha^\vee_5+4\alpha^\vee_6+ 2\alpha^\vee_7. $$
  In this case, $m_{\omega_6^\vee}= 2$.   By the computation in \cite[\S3, Case $E_7$]{Kl} (note that the module $W_1(6)$ considered there is isomorphic to our Demzure module $D(1,\omega^\vee_6)$), the adjoint representation $\mf{g}\simeq V_{\omega_1}$ also appears in $D(1, \omega_6^\vee)$ with multiplicity one. By the similar argument as in the previous case,  $T_e\Gr_{\preceq \omega_6^\vee}\simeq \mf{g} t^{-1}$.   Therefore,   $\Gr_{\preceq \omega_6^\vee}\not= \FM_{\preceq \omega_6^\vee}$. 
  
\begin{remark}
When $G$ is of type $E_6$,  $\omega^\vee_1$ and $\omega^\vee_6$ are minuscule. Our example above shows that when $\lambda$ is a summation of minuscule coweights, the reducedness conjecture still could fail. 
  
\end{remark}

\end{document}